\documentclass[12pt, a4paper]{amsart}
\usepackage{amsmath}
\usepackage{amsfonts}
\usepackage{amssymb}
\usepackage{amscd}
\usepackage{graphicx}
\theoremstyle{definition}

\theoremstyle{definition}

\theoremstyle{definition}

\theoremstyle{definition}

\theoremstyle{definition}

\numberwithin{equation}{section}

\newtheorem{definition}{Definition}[section]
\newtheorem{lemma}{Lemma}[section]
\newtheorem{theorem}{Theorem}[section]
\newtheorem{conjecture}{Conjecture}[section]
\newtheorem{corollary}{Corollary}[section]
\newtheorem{remark}{Remark}[section]
\theoremstyle{remark}

\begin{document}
\title{On Three-dimensional CR Yamabe Solitons}
\author{Huai-Dong Cao$^{1}$}
\address{$^{1}$Department of Mathematics, University of Macau, Macao \& Lehigh University, Bethlehem, PA
18015, USA}
\email{huc2@lehigh.edu }
\author{$^{\ast }$Shu-Cheng Chang$^{2}$}
\address{$^{2}$Department of Mathematics and Taida Institute for
Mathematical Sciences (TIMS), National Taiwan University, Taipei 10617,
Taiwan, R.O.C.}
\email{scchang@math.ntu.edu.tw }
\author{$^{\ast }$Chih-Wei Chen$^{3}$}
\address{$^{3}$Department of Mathematics, National Taiwan University, Taipei
10617, Taiwan, R.O.C. }
\email{BabbageTW@gmail.com}
\thanks{$^{\ast}$Research supported in part by the NSC of Taiwan}
\subjclass[2010]{Primary 32V05; Secondary 53C12}
\keywords{CR Harnack quantity, CR Yamabe soliton, CR
Paneitz operator.}

\begin{abstract}
In this paper, we investigate the geometry and classification of three-dimensional CR Yamabe solitons. 
In the compact case, we show that any $3$-dimensional CR Yamabe soliton must have constant 
Tanaka-Webster scalar curvature; we also obtain a classification under the assumption that their potential functions are in the kernel of the CR Paneitz operator. 
In the complete case, we obtain a structure theorem on the diffeomorphism types of complete $3$-dimensional pseudo-gradient CR Yamabe solitons (shrinking, or steady, or expanding) of vanishing torsion.
\end{abstract}

\maketitle

\section{Introduction}

Inspired by Perelman's work \cite{p1, p2} on Hamilton's Ricci flow (cf. \cite{h1}-\cite{h5}), self-similar solutions, also known as (geometric) solitons, of various geometric flows have attracted a lot of attentions in recent years because of their close ties with singularity formations in the geometric flows.  In particular, important progress has been made in the study of Ricci solitons, 
self-similar solutions to the mean curvature flow, as well as Yamabe solitons, etc. 

For pseudohermitian manifolds, similar to the concept of Yamabe solitons, one can introduce the notion of CR Yamabe solitons
and pseudo-gradient CR Yamabe solitons.

\begin{definition} A pseudohermitian $(2n+1)$-manifold $(M^{2n+1}, J,\theta )$, with CR structure $J$ and compatible contact form $\theta$, is called a CR Yamabe soliton if there exist an 
infinitesimal contact diffeomorphism $X$ and a constant 
$\mu\in {\mathbb R}$  such that  
\begin{equation}\label{CRYamabe}
\left \{ 
\begin{array}{l}
W\theta+\frac{1}{2}L_X\theta =\mu\theta,\\
L_{X}J=0,
\end{array}
\right. 
\end{equation}
where 
$W$ is the Tanaka-Webster scalar curvature of $(M^{2n+1}, J,\theta)$ and $L_{X}$ denotes Lie derivative by $X$. It is called {\it shrinking} if $\mu>0$, {\it steady}  if $\mu=0$, and {\it expanding}  if $\mu<0$. 
\end{definition}

\begin{remark}  An infinitesimal contact diffeomorphism is a vector field $X$ on $M$ such that  $L_{X}\theta=\lambda\theta $ for some function $\lambda$. It is known that for any infinitesimal contact diffeomorphism $X$, there always associates a function $f$ such that $X=J(\nabla_b f)+f\bf{T}$ and $L_{X}\theta=f_0\theta$,
where $\nabla_b$ is the subgradient, and the subscript $0$ denotes the differentiation along the Reeb vector field $\bf{T}$. 
Note that the first equation in (1.1) can then be expressed as $W+\frac{1}{2}f_0=\mu$, while the second equation $L_{X}J=0$ is corresponding to $f_{\alpha\alpha}+iA_{\alpha\alpha} f=0$ with torsion tensor $A=\{A_{\alpha\alpha}\}$. We refer the reader to Section 3 for more details.
\end{remark}

As we shall see, CR Yamabe solitons also correspond to self-similar solutions to the CR Yamabe flow on a pseudohermitian $(2n+1)$-manifold $(M^{2n+1}, J,\mathring{\theta})$ given by (cf. \cite{cc})
\begin{equation}
\left \{ 
\begin{array}{l}
\frac{\partial }{\partial t}\theta(t) =-2W(t)\theta \left( t\right) , \\ 
\theta \left( 0\right) =\mathring{\theta},
\end{array}
\right.  \label{1b}
\end{equation}%
where $W(t)$ is the Tanaka-Webster 
scalar curvature with respect to the evolving contact form $\theta \left( t\right)$.
Like the Riemannian Yamabe flow, which has been very well understood in the compact case, 
the compact CR Yamabe flow hopefully provides a canonical deformation from the given contact form 
$\mathring{\theta}$ to a contact form $\theta$ of constant Tanaka-Webster scalar curvature.
As noted in \cite{cc}, the CR Yamabe flow has a unique short time solution on any compact 
pseudohermitian $(2n+1)$-manifold $(M^{2n+1}, J,\mathring{\theta})$. 
Moreover, on a compact pseudohermitian $3$-manifold $(M^3 J, \mathring{\theta})$, 
solutions to the normalized CR Yamabe flow 
\begin{equation}
\frac{\partial }{\partial t}\theta (t)=2(r(t)-W(t))\theta (t)\ ;\  \
r(t)=\int_{M}W(t)d\mu /\int_{M}d\mu   \label{2011a}
\end{equation}
exist for all time. This long time existence of the three-dimensional normalized CR 
Yamabe flow was established in \cite{h6} and \cite{ss}.
However, unlike the Riemannian Yamabe flow, the problem of asymptotic convergence of solutions
of the CR Yamabe flow is widely open even on closed pseudohermitian $3$-manifolds.
Here we propose the following
\begin{conjecture}
\label{conj} 
Let $(M^3, J,\mathring{\theta})$ be a closed spherical CR $3$-manifold with positive Tanaka-Webster scalar curvature 
and vanishing torsion. Then, as $t\rightarrow \infty$, the solution to the normalized CR Yamabe flow converges smoothly to a unique limit 
contact form of positive constant Tanaka-Webster scalar curvature and vanishing 
pseudohermitian torsion.
\end{conjecture}

A related result was derived in a paper of the second named author with others \cite{ccw}, 
where they proved the stability property for solutions of the normalized CR Yamabe flow 
near the standard CR $3$-sphere $(\mathbf{S}^{3},\widehat{J},\widehat{\theta })$ which is 
spherical with positive constant Tanaka-Webster curvature and vanishing pseudohermitian torsion. 
We also refer to \cite{ho} for another result related to this conjecture.
 
We remark that one of the difficulties is that CR Yamabe flow does not improve the behavior 
of the pseudohermitian torsion. So another flow, called the CR torsion flow, was 
considered by the second named author and his coauthors (\cite{ckw}) on a pseudohermitian $3$-manifold: 
the {\it CR torsion flow} is defined by
\begin{equation}
\left \{ 
\begin{array}{l}
\frac{\partial }{\partial t}J(t)=-2JA_{J,\theta }(t),\\ 
\frac{\partial }{\partial t}\theta (t)=-2W(t)\theta (t),
\end{array}
\right.   \label{2011aaa}
\end{equation}
on $M\times \lbrack 0,T)$ with the CR structure $J(t)=i\theta^{1}\otimes
Z_{1}-i\theta ^{\bar{1}}\otimes Z_{\bar{1}}$ and the
pseudohermitian torsion tensor $A_{J,\theta }(t)=-iA_{11}\theta ^{1}\otimes
Z_{\bar{1}}+iA_{\bar{1}\bar{1}}\theta ^{\bar{1}}\otimes
Z_{1}.$ 
This flow is a CR analogue of the Ricci flow, and the CR Yamabe flow (\ref{1b}) 
is a special formulation of it. 

On the other hand, in view of the work of \cite{ds}, \cite{csz}, and \cite{cmm}, understanding the structure of CR Yamabe solitons may be a necessary step in approaching the asymptotic convergence of solutions 
of the CR Yamabe flow (\ref{1b}). Indeed, one expects CR Yamabe solitons to model 
singularity formations of the CR Yamabe flow.

In this paper, we investigate the geometry and classification of three-dimensional closed CR Yamabe solitons 
 $\left( M^{3}, J, \theta, f, \mu\right)$ 
satisfying the following equivalent soliton equations
\begin{equation*}
\left \{ 
\begin{array}{l}
W+\frac{1}{2}f_{0}=\mu , \\ 
f_{11}+iA_{11}f=0,%
\end{array}%
\right. 
\end{equation*}%
for some smooth function $f$ on a pseudohermitian $3$-manifold $(M^3,  J, \theta)$. 


Our first result asserts that any closed three-dimensional CR Yamabe soliton must 
have constant Tanaka-Webster scalar curvature. 

\begin{theorem}
\label{A} If $\left( M^{3}, J, \theta , f, \mu \right) $ is a closed three-dimensional CR
Yamabe soliton, then $\theta $ is the contact structure of constant
Tanaka-Webster scalar curvature $\mu$, i.e., $f_{0}=0$.
\end{theorem}

Our second result establishes the vanishing of the pseudohermitian torsion of a closed 
three-dimensional CR Yamabe soliton under certain conditions. 

\begin{theorem}
\label{B} Let $\left( M^{3},J,\theta , f, \mu \right) $ be a closed CR
Yamabe soliton with $P_{0}f=0$ for the CR Paneitz operator $P_{0}$. If the contact class $[\theta ]$ of $\theta$ 
admits a contact form $\widetilde{\theta }$ of the vanishing CR $\widetilde{Q}$%
-curvature, then $(M,J,\theta )$ is a closed pseudohermitian $3$-manifold of
vanishing pseudohermitian torsion provided that $f$ is nowhere vanishing. 
\end{theorem}



\begin{remark}

(i). The kernel space of CR Paneitz operator $P_{0}$ is
infinite-dimensional.

(ii). Since a pseudohermitian $3$-manifold $(M,J,\theta)$ of constant Tanaka-Webster scalar curvature 
and vanishing pseudohermitian torsion must be spherical (see Definition 2.1 and Eq. (\ref{spherical})), 
it follows from Theorem 1.1, Theorem 1.2 and a result of Y. Kamishima and T. Tsuboi (\cite{kt})  
that one can have a complete classification of such closed spherical torsion-free CR Yamabe 
solitons.

(iii). It is conjectured that any closed three-dimensional pseudohermitian manifold admits 
a contact form of vanishing CR $Q$-curvature (cf. \cite{fh, ccc, cs}). 
Therefore, in view of Theorem \ref{B}, we have
\begin{conjecture}\label{conj2}
Any $3$-dimensional closed CR Yamabe soliton $\left( M^{3},J,\theta ,f,\mu \right)$ 
with $P_{0}f=0$ and nowhere vanishing $f$ must have zero pseudohermitian torsion. 
\end{conjecture}
\end{remark}

As a consequence of Theorem \ref{B} and the work of J. Cao and the second named author \cite{caoc},  
we have the following classification of a special class of closed CR Yamabe solitons:

\begin{corollary}
\label{C1} 
Conjecture \ref{conj2} holds if $(M^3, J, \theta)$ is the smooth boundary of a bounded strictly pseudoconvex domain $\Omega$ 
in a complete Stein manifold $V^{2}$. 
\end{corollary}

In particular,  if $M$ is a hypersurface in $\mathbb{C}^{2}$, i.e. $M=\partial
\Omega $ for a bounded domain $\Omega $ in $\mathbb{C}^{2}$, then for any
pluriharmonic function $u:\mathbf{U}\rightarrow \mathbb{R}$ ($\partial 
\overline{\partial }u=0$) with a simply connected $\mathbf{U}\subset 
\overline{\Omega }$, there exists a holomorphic function $w$ in $\mathbf{U}$ 
such that $u=\mathrm{Re}(w)$.

Now define $f:=u|_{M}$, it follows that $f$ is a CR pluriharmonic function 
(see Definition \ref{4}) and  
\begin{equation*}
P_{0}f=0.
\end{equation*}

Hence, we have

\begin{corollary}
\label{C2} Let $(M,J,\theta )$ be a smooth hypersurface in $\mathbb{C}^{2}$.
If $\left( M^{3},J,\theta, f, \mu \right) $ is a closed CR Yamabe soliton
with $f:=u|_{M}$ as above which vanishes nowhere. 
Then $(M,J,\theta )$ is a closed pseudohermitian $%
3$-manifold of constant Tanaka-Webster scalar curvature $\mu $ and vanishing
pseudohermitian torsion.
\end{corollary}

Secondly, we consider complete three-dimensional pseudo-gradient CR Yamabe solitons defined as follows: 

\begin{definition} A complete pseudohermitian $3$-manifold $(M^{3},  J, \theta )$ is called a complete {\it pseudo-gradient} CR 
Yamabe soliton if there exists a smooth function $\varphi $ on $M$ such that 

\begin{equation*}
\left \{ 
\begin{array}{l}
W+\frac{1}{2}\Delta _{b}\varphi =\mu , \\ 
\varphi _{11}=0,\  \varphi _{0}=0.%
\end{array}%
\right.
\end{equation*}%
Again, it is called {\it shrinking} if $\mu>0$, {\it steady} if $\mu=0$, and {\it expanding} if $\mu<0$. 
\end{definition}

\medskip
For example, by using Lemma \ref{l52}, one can easily check that the
Heisenberg group $$H^3\equiv \{(z,t)\in \mathbb{C}\times\mathbb{R}\ |\ (z_1,t_1)\cdot(z_2,t_2):=(z_1+z_2, t_1+t_2+\mathrm{Im}(z_1\bar{z_2}) ) \},$$ equipped with the standard CR structure and potential function 
$\varphi:= \mu |z|^2$, is a pseudo-gradient CR Yamabe soliton for any $\mu\in \mathbb{R}$. Since it is analogous to the Gaussian Ricci solitons defined on the Euclidean space ${\mathbb R}^n$ with the potential function $\mu |x|^2$, we call  $(H^3, \mu |z|^2)$ a {\it pseudo Gaussian soliton}.
Note that the Heisenberg group also admits CR Yamabe soliton structures of shrinking/steady/expanding type, as defined in (\ref{CRYamabe}), by taking $f=2\mu t$.

For three-dimensional complete pseudo-gradient CR Yamabe solitons, 
motivated by the recent work \cite {csz} of the first author with others on the structure 
of gradient Yamabe solitons, we explore the potential function $\varphi$ and investigate the 
geometry and topology of its level sets. It turns out that $\varphi$ is necessarily an {\sl isoparametric} function which, according to \cite{wang}, makes the critical set of $\varphi$ and regular 
level sets rather special (see Lemma \ref{isop}  and Remark \ref{isopara}): the only singular level sets of $\varphi$ are the smooth focal varieties of $\varphi$ (possibly empty), 
$$\Sigma_+:= \{\varphi = \max_{x\in M}\varphi(x)\}\ \mbox{ and }\ \Sigma_-:= \{\varphi = \min_{x\in M}\varphi(x)\},$$ and each regular level set is a tube over either $\Sigma_+$ or $\Sigma_{-}$.  Furthermore, we prove that each regular level surface must have zero Gaussian curvature with respect to the induced metric of the Webster adapted metric on $M^3$. 
These special geometric and topological features allow us to 
conclude the possible diffeomorphism types of the underlying $3$-manifold $M$.

\begin{theorem}
\label{E} 
Let $(M^{3},J,\theta,\varphi,\mu )$ be a non-trivial three-dimensional complete 
pseudo-gradient CR Yamabe soliton with vanishing torsion. 
Then the potential function $\varphi$ is an isoparametric function and $M$ is diffeomorphic to one of the following spaces:









\smallskip

$\mathbb{R}^3$, $\mathbb{S}^3$, $L(p,q)$, $\mathbb{S}^2\times \mathbb{R}$, $\mathbb{S}^1\times\mathbb{R}^2$, $\mathbb{T}^2\times \mathbb{R}$, $\mathbb{T}^2\times [0,\infty)\mbox{ with }\mathbb{T}^2\times \{0\} \mbox{ collapsing to }\mathbb{S}^1$.

\smallskip
\noindent More precisely, besides possible surface components, the critical set of $\varphi$ can contain at most two curves. Furthermore, 

\smallskip

(i) if the critical set is empty or consists of only surfaces, then $M$ is diffeomorphic to either $\mathbb{R}^3$, or $\mathbb{T}^2\times \mathbb{R}$,  or $S^1 \times \mathbb{R}^2$, where $\mathbb{T}^2$ denotes the 2-torus;

(ii) if the critical set contains only one curve, 
then $M$ is diffeomorphic to either $\mathbb{R}^3$ or $\mathbb{T}^2\times [0,\infty)\mbox{ with }\mathbb{T}^2\times \{0\} \mbox{ collapsing to }S^1$; 

(iii) if the critical set contains two curves, 
then $M$ is diffeomorphic to either $\mathbb{S}^3$, or $\mathbb{S}^2\times \mathbb{R}$, or the lens spaces $L(p,q)$ with $1\leq q<p$. 

\end{theorem}

\smallskip
Moreover, when $M^3$ is simply-connected, we can further show the following
\begin{corollary}
\label{F} 
Let $(M^{3},J,\theta,\varphi,\mu)$ be a simply-connected complete   
pseudo-gradient CR Yamabe soliton with vanishing torsion and $W>\mu$ (or $W<\mu$), then it must be trivial or diffeomorphic to $\mathbb{R}^3$.
\end{corollary}

\begin{remark}
M. Rumin (\cite{ru}) has proven that a complete pseudohermitian 
$3$-manifold of positive Tanaka-Webster 
scalar curvature and vanishing torsion must be compact.
Hence, if $W\geq \mu>0$, then the soliton is trivial and $W=\mu$. 
In fact, the manifold is the CR $3$-sphere (cf. \cite{kt}).
\end{remark}

The rest of the paper is organized as follows. 
In Section $2$, we review some basic materials of pseudohermitian manifolds 
and recall some basic facts from the CR Yamabe flow. 
In Section $3$, we first explain that CR Yamabe solitons correspond to self-similar solutions 
of the CR Yamabe flow and then prove Theorem 1.1, Theorem 1.2 and the two corollaries.
Complete pseudo-gradient CR Yamabe solitons are defined in Section $4$. Some basic properties
are also derived there.  
In Section $5$, we investigate the structure of pseudo-gradient CR Yamabe solitons  
with vanishing torsion. 

\bigskip
\noindent \textbf{Acknowledgements.} We would like to thank Professor Feng Luo for very helpful discussions related to Theorem 1.3 and Professor Reiko Miyaoka for her comment during the 7th OCAMI-TIMS Workshop which led us to remove an extra assumption of Theorem 1.3 in an early version.
The research of the first author was partially supported by the Science and Technology Development Fund (Macao S.A.R.) Grant FDCT/016/2013/A1, as well as RDG010 grant of University of Macau. Part of the project was done during the visit of the second and the third authors to the University of Macau in spring 2014. They would like to express their thanks to the institution for the warm hospitality.

\section{Preliminaries}

In this section, we first introduce some basic materials about 
pseudohermitian $3$-manifolds, and then collect some basic facts about the CR
Yamabe flow. We refer the reader to \cite{ccw} and \cite{cc} for
more details and \cite{l1,l2} for higher dimensional cases.

Let $M$ be a complete $3$-manifold with an oriented contact structure $\xi $.
There always exists a global contact form $\theta $ with $\xi =\ker \theta $, 
obtained by patching together local ones with a partition of unity. 
The unique vector field $\mathbf{T}$ such that ${\theta }(\mathbf{T})=1$ and 
$L_{\mathbf{T}}{\theta }=0$ or $d{\theta}(\mathbf{T},{\cdot })=0$
is called the Reeb vector field of $\theta $. 
A CR structure compatible with $\xi $ is a smooth endomorphism 
$J:{\xi }{\rightarrow }{\xi }$ such that $J^{2}=-Id$. 
A pseudohermitian structure compatible with $\xi$ is a CR-structure $J$
compatible with $\xi $ together with a global contact form $\theta $. The CR
structure $J$ extends to $\mathbb{C}\otimes \xi $ and decomposes 
$\mathbb{C}\otimes \xi $ into the direct sum of $T_{1,0}$ and $T_{0,1}$ which are
eigenspaces of $J$ with respect to eigenvalues $i$ and $-i$, respectively.

Let $\left \{ \mathbf{T},Z_{1},Z_{\bar{1}}\right \} $ be a frame of $TM\otimes \mathbb{C}$, 
where $Z_{1}\in T_{1,0}$ and $\ Z_{\bar{1}}=\overline{Z_{1}}\in T_{0,1}$, 
and $\left \{ \theta ,\theta^1,\theta ^{\bar{1}}\right \} $ be the coframe dual to it.
Since $d\theta(\mathbf{T}, \cdot)=0$, there exists a positive function $h_{1\bar{1}}$ such that
\begin{equation}
d\theta =ih_{1\bar{1}}\theta ^{1}\wedge \theta ^{\bar{1}}.  \label{100242}
\end{equation}
In this article, we always normalize $Z_{1}$ such that $h_{1\bar{1}}=1$ and
denote $d\mu:= \theta\wedge d\theta$ as the volume form of $M$.

Now we introduce the Levi form $\left \langle \ ,\  \right \rangle _{L_{\theta }}$, which is defined by  
\begin{equation*}
\left \langle V,W\right \rangle _{L_{\theta }}=-i d\theta \left (V\wedge \overline{W}\right ) = d\theta  \left (V\wedge J\overline{W}\right ),
\end{equation*}%
where $V,W\in T_{1,0}$. By defining $\left \langle \overline{V},\overline{W}\right \rangle_{L_{\theta }} =\overline{\left\langle V,W\right \rangle_{L_{\theta }} }$ for all $V,W\in T_{1,0}$,
one can extend the Levi form to $T_{0,1}$. In fact, the Levi form can be further 
extended as a hermitian form on all the tensor bundles composed of $T_{0,1},T_{1,0}$ and their duals.

Let $\theta _{1}{}^{1}$ be the unique purely imaginary $1$-form such that
\begin{equation}
\begin{split}
d\theta ^{1}& =\theta ^{1}\wedge \theta _{1}{}^{1}+\theta \wedge \tau ^{1}, \\
\end{split}
\label{100240}
\end{equation}
where $\tau ^{1}=A^{1}{}_{\bar{1}}\theta ^{\bar{1}}$ is the pseudohermitian torsion.
Then the 1-forms $\theta _{1}{}^{1}$ and $\theta _{\bar{1}}{}^{\bar{1}}:=\overline{\theta _{1}{}^{1}}$, 
as connection coefficients, characterize the pseudohermitian connection $\nabla$ of $(M,J,\theta)$. 
Namely, 
\begin{equation*}
\nabla Z_{1}=\theta _{1}{}^{1}\otimes Z_{1},\quad \nabla Z_{\bar{1}}=\theta
_{\bar{1}}{}^{\bar{1}}\otimes Z_{\bar{1}}\ \ \mbox{ and }\ \ \nabla \mathbf{T}=0.
\end{equation*}%
Note that $\nabla$ can be extended to all tensors naturally. 
Furthermore, we have the following structure equation
\begin{equation}
d\theta _{1}{}^{1}=W\theta ^{1}\wedge \theta ^{\bar{1}}+2i\mathrm{Im}(A^{\bar{1%
}}{}_{1,\bar{1}}\theta ^{1}\wedge \theta ) \label{100241}
\end{equation}%
and $W$ is called the Tanaka-Webster curvature.
Here we denote components of covariant derivatives with indices preceded by
comma; thus write $A^{\bar{1}}{}_{1,\bar{1}}\theta ^{1}\wedge \theta $. The
indices $\{0,1,\bar{1}\}$ indicate derivatives with respect to $\{\mathbf{T},Z_{1},Z_{%
\bar{1}}\}$. For derivatives of a scalar function, we often omit the
comma, for instance, $\varphi _{1}=Z_{1}\varphi ,\  \varphi _{1\bar{1}}=Z_{%
\bar{1}}Z_{1}\varphi -\theta _{1}{}^{1}(Z_{\bar{1}})Z_{1}\varphi ,\  \varphi
_{0}=\mathbf{T}\varphi $ for a (smooth) function $\varphi $. 

For a real-valued function $\varphi $, the subgradient $\nabla _{b}$ is defined by 
$\nabla _{b}\varphi \in \xi $ and $\left \langle V,\nabla _{b}\varphi
\right
\rangle_{L_\theta} =d\varphi (V)$ for all vector fields $V$ tangent to contact
plane. Locally $\nabla _{b}\varphi =\varphi _{\bar{1}}Z_{1}+\varphi _{1}Z_{%
\bar{1}}$. We can use the connection to define the subhessian as the complex
linear map $(\nabla ^{H})^{2}\varphi :T_{1,0}\oplus T_{0,1}\rightarrow
T_{1,0}\oplus T_{0,1}$ with 
\begin{equation*}
(\nabla ^{H})^{2}\varphi (V)=\nabla _{V}\nabla _{b}\varphi .
\end{equation*}%
The sub-Laplacian $\Delta _{b}$ is defined to be the trace of the subhessian%
\begin{equation*}
\Delta _{b}\varphi = \mathrm{tr}\left( (\nabla ^{H})^{2}\varphi \right) 
=\varphi _{1\bar{1}}+\varphi _{\bar{1}1}.
\end{equation*}%

For all $V=V^{1}Z_{1}\in T_{1,0}$, we define
\begin{equation*}
\begin{split}
Ric(V,V)& =WV^{1}V^{\bar{1}}=W|V|_{L_{\theta }}^{2}, \\
Tor(V,V)& =2Re\ iA_{\bar{1}\bar{1}}V^{\bar{1}}V^{\bar{1}}.
\end{split}%
\end{equation*}

We recall the following commutation relations (\cite{l1}).%
\begin{equation}
\begin{array}{ccl}
C_{I,01}-C_{I,10} & = & C_{I,\bar{1}}A_{11}-kC_{I,}A_{11,\bar{1}},
\\ 
C_{I,0\bar{1}}-C_{I,\bar{1}0} & = & C_{I,1}A_{\bar{1}%
\bar{1}}+kC_{I,}A_{\bar{1}\bar{1},1}, \\ 
C_{I,1\bar{1}}-C_{I,\bar{1}1} & = & iC_{I,0}+kWC_{I},
\end{array}
\label{2010a}
\end{equation}
here $C_{I}$ denotes a coefficient of a tensor with multi-index $I$
consisting of only $1$ and $\bar{1}$, and $k$ is the number of $1$'s
minus the number of $\bar{1}$'s in $I$.

Next we rewrite the purely imaginary $1$-forms $\theta _{1}{}^{1}=i\sigma
_{1}^{2}$ and $\theta _{\bar{1}}{}^{\bar{1}}=i\sigma _{2}^{1}$ in terms of
real $1$-forms $\sigma_1^2$ and $\sigma_2^1$ with $\sigma _{2}^{1}=-\sigma _{1}^{2}\ 
$by (\ref{100240}). Let $Z_{1}=\frac{1}{2}(e_{1}-ie_{2})$ for real vectors $%
e_{1}$ and $e_{2}.$ It follows that $e_{2}=Je_{1}$. Let $e^{1}=\mathrm{Re}(\theta
^{1})$ and $e^{2}=\mathrm{Im}(\theta ^{1})$. Then $\{e^{1},e^{2},\theta \}$ is
dual to $\{e_{1},e_{2},\mathbf{T}\}$. Now in view of (\ref{100240}) and (\ref{100242}%
), we have the following real version of structure equations (\cite{cchi}) :%
\begin{align*}
d\theta &=2e^{1}\wedge e^{2}, \\
\nabla _{b}e_{1} &=\sigma _{1}^{2}\otimes e_{2},\text{ }\nabla
_{b}e_{2}=\sigma _{2}^{1}\otimes e_{1}, \\
de^{1} &=e^{2}\wedge \sigma _{2}^{1}\text{ mod }\theta ;\text{ }%
de^{2}=e^{1}\wedge \sigma _{1}^{2}\text{ mod }\theta .
\end{align*}%
We also write $\varphi _{e_{i}}=e_{i}\varphi $ and $\nabla _{b}\varphi =%
\frac{1}{2}(\varphi _{e_{1}}e_{1}+\varphi _{e_{2}}e_{2}).$ Moreover we have 
\begin{equation}
\varphi _{e_{i}e_{j}}=e_{j}e_{i}\varphi -\sigma _{i}^{k}(e_{j})\varphi
_{e_{k}}  \label{2014da}
\end{equation}
for $i,\ j,\ k=1,\ 2$ \ and%
\begin{equation}
\Delta _{b}\varphi =(\varphi _{1\bar{1}}+\varphi _{\bar{1}1})=\frac{1}{2}%
(\varphi _{e_{1}e_{1}}+\varphi _{e_{2}e_{2}}).  \label{2014d}
\end{equation}
The real version of the commutation relations are
\begin{equation}
\begin{array}{lcl}
\varphi _{e_{1}e_{2}}-\varphi _{e_{2}e_{1}} & = & 2\varphi _{0} \\ 
\varphi _{0e_{1}}-\varphi _{e_{1}0} & = & \varphi _{e_{1}}\text{Re}%
A_{11}-\varphi _{e_{2}}\text{Im}A_{11} \\ 
\varphi _{0e_{2}}-\varphi _{e_{2}0} & = & \varphi _{e_{1}}\text{Im}%
A_{11}+\varphi _{e_{2}}\text{Re}A_{11} \\ 
\varphi _{e_{1}e_{1}e_{2}}-\varphi _{e_{1}e_{2}e_{1}} & = & 2\varphi
_{e_{1}0}-2\varphi _{e_{2}}W \\ 
\varphi _{e_{2}e_{1}e_{2}}-\varphi _{e_{2}e_{2}e_{1}} & = & 2\varphi
_{e_{2}0}+2\varphi _{e_{1}}W.%
\end{array}
\label{2010b}
\end{equation}

We also recall some definitions for later purposes.

\begin{definition}[\cite{cl,kt}]
Let $(M,J,\theta)$ be a three-dimensional pseudohermitian manifold. 
Then the followings are equivalent:

(i) Its Cartan curvature tensor vanishes, i.e.,
\begin{equation}\label{spherical}
Q_{11}=\frac{1}{6}W_{11}+\frac{i}{2}WA_{11}-A_{11,0}-\frac{2i}{3}A_{11,\bar{1}1}=0.
\end{equation}

(ii) It is locally CR equivalent to the standard pseudohermitian $3$-sphere $(\mathbf{S}^{3},\widehat{J},\widehat{\theta })$. 

If such properties hold, then we call the CR structure $J$ spherical 
and $(M,J,\theta)$ a spherical pseudohermitian $3$-manifold. 
In particular, the spherical structure is CR invariant.

\end{definition}

\begin{definition}
Let $(M, J, \theta )$ be a three-dimensional pseudohermitian manifold without boundary.
A piecewise smooth curve $\gamma :[0,1]\rightarrow M$ is said to be a
Legendrian curve if $\gamma \, ^{\prime }(t)\in \xi $ whenever $\gamma \,
^{\prime }(t)$ exists. The length of $\gamma $ is then defined by 
\begin{equation*}
l(\gamma )=\int_{0}^{1}h(\gamma\,^{\prime }(t),\gamma \, ^{\prime }(t))^{%
\frac{1}{2}}dt,
\end{equation*}%
where $h(X, Y)=d\theta (X, JY)$. 
The Carnot-Carath\'{e}odory distance $d_{c}$ between any two points $p,q\in M$
is defined by 
\begin{equation*}
d_{c}(p,q)=\text{inf}\left \{ l(\gamma )|\  \gamma \in C_{p,q}\right \} ,
\end{equation*}%
where $C_{p,q}$ is the set of all Legendrian curves which join $p$ and $q$.
By Chow's connectivity theorem \cite{cho}, for any two points $p, q\in M$, there always exists a Legendrian
curve joining $p$ and $q$, so the distance is finite. We say $(M, J, \theta )$ is complete
if $(M,d_{c})$ is a complete metric space. 
\end{definition}

Finally, in the following, we recall some basic facts about the CR Yamabe flow (\ref{1b}) 
on a closed pseudohermitian $3$-manifold $\left( M^{3},J,\mathring{\theta}\right)$ with 
$\theta (0)=\mathring{\theta}$.
From (\cite{cc}), we have

\begin{lemma}
\label{lem1} Let $\left( M^{3}, J, \mathring{\theta}\right) $
be a closed pseudohermitian $3$-manifold. Under the CR Yamabe flow (\ref{1b}%
), we have%
\begin{equation}
\frac{\partial }{\partial t}W=4\Delta _{b}W+2W^{2}  \label{10110}
\end{equation}%
and 
\begin{equation}
\frac{\partial }{\partial t}A_{11}=2WA_{11}-2iW_{11}.  \label{10115}
\end{equation}%
Here $\Delta_{b}$, $A_{11}$ and the covariant derivatives are with respect to $%
\theta \left( x,t\right) $ which is in the contact class $[\mathring{\theta}%
(x)]$ with $\theta (x,0)=\mathring{\theta}(x){.}$
\end{lemma}

Now it follows from applying the maximum principle to (\ref{10110})  that

\begin{lemma}
\label{lem2} Let $\left( M^{3},J,\mathring{\theta}\right) $
be a closed pseudohermitian $3$-manifold with positive initial
Tanaka-Webster curvature $\mathring{W}>0$. Then 
\begin{equation*}
W(t) >0
\end{equation*}%
is preserved under the CR Yamabe flow (\ref{1b}).
\end{lemma}

Moreover, if $J$ is spherical so that %
\begin{equation*}
Q_{11}=\frac{1}{6}W_{11}+\frac{i}{2}WA_{11}-A_{11,0}-\frac{2i}{3}A_{11,\bar{1%
}1}=0,
\end{equation*}%
i.e., 
\begin{equation}
\frac{1}{6}W_{11}=-\frac{i}{2}WA_{11}+A_{11,0}+\frac{2i}{3}A_{11,\bar{1}1}, 
\label{10120}
\end{equation}%
then, by using this and the commutation relation, one can rewrite the RHS of (\ref%
{10115}) as%
\begin{equation*}
4\left( A_{11,\bar{1}1}+A_{11,1\bar{1}}-4iA_{11,0}\right) -12WA_{11}.
\end{equation*}
Thus, we have%
\begin{equation*}
\frac{\partial }{\partial t}A_{11}=-4\mathcal{L}_{4}A_{11}-12WA_{11}, 
\end{equation*}%
so the highest ``weight" term is just $-4$ times the generalized Folland-Stein
operator 
\begin{equation*}
\mathcal{L}_{\alpha }=-\Delta _{b}+i\alpha \mathbf{T}
\end{equation*}
acting on $A_{11}$ with $\alpha =4$. Since $\alpha =4$ is clearly not an odd integer, $-%
\mathcal{L}_{4}$ is subelliptic on a closed pseudohermitian $3$-manifold. It
follows from the standard theory for subelliptic equations that we have

\begin{lemma}
\label{lem3} \ Let $\left( M^{3},J,\mathring{\theta}\right) $
be a closed spherical pseudohermitian $3$-manifold with vanishing initial
torsion $\mathring{A}_{11}=0$. Then%
\begin{equation*}
A_{11}\left( t\right) =0
\end{equation*}%
is preserved under the CR Yamabe flow (\ref{1b}).
\end{lemma}

\section{The Structure of Closed three-dimensional CR Yamabe Solitons}

In this section, by using the Harnack quantity derived in (\ref{3}), we show that every
closed three-dimensional CR Yamabe soliton has constant Tanaka-Webster curvature.

We first recall a result from \cite{g} (see also \cite{cl}, Lemma 3.4 and Lemma 3.5).

\begin{lemma}\label{l41} 
Let $\left( M^{2n+1}, J, \theta \right) $ be a pseudohermitian $2n+1$-manifold. 
For any smooth function $f$ on $M$, let $X_{f}$ be the vector field uniquely defined by 
\begin{equation}
X_{f}\rfloor \ d\theta =df\text{ mod }\theta  \label{41a}
\end{equation}%
and 
\begin{equation}
\ X_{f}\rfloor \theta =-f.  \label{41c}
\end{equation}%
Then 
\begin{equation}
X_{f} =if_{\alpha}Z_{\bar{\alpha}}-if_{\bar{\alpha}}Z_{\alpha}-f\mathbf{T}  \label{41b}
\end{equation}%
and it is a smooth infinitesimal contact diffeomorphism of 
$\left( M^{2n+1}, \theta \right)$. 
Conversely, every smooth infinitesimal contact diffeomorphism is of the
form $X_{f}$ for some smooth function $f$. Moreover, $L_{X_{f}}J$ has the following expression
\begin{equation*}
L_{X_{f}}J \equiv 2(f_{\alpha\alpha}+iA_{\alpha\alpha}f)\theta^{\alpha}\otimes Z_{\bar{\alpha}
}+2(f_{\bar{\alpha}\bar{\alpha}}-iA_{\bar{\alpha}\bar{\alpha}}f)\theta^{\bar{\alpha}}\otimes Z_{\alpha}\  \mathrm{mod }\ \theta .
\end{equation*}
\end{lemma}

\begin{remark}
\label{r41} Note that, since $L_{\mathbf{T}}\theta =0$, 
$L_{X_{f}-c\mathbf{T}}\theta=\lambda \theta $ for any constant $c$ if $L_{X_{f}}\theta =\lambda \theta $.
Hence if $X_{f}$ is an infinitesimal contact diffeomorphism, 
so is $X_{\widetilde{f}}:=X_{f}-c\mathbf{T}$ with $\widetilde{f}=f+c$.
\end{remark}

\begin{definition}
A pseudohermitian $3$-manifold $(M^{3}, J,\theta )$ is called a CR Yamabe soliton if there exist a function $f$ and a constant $\mu$ such that 
\begin{equation}
W+\frac{1}{2}f_{0}=\mu  \label{1}
\end{equation}%
and 
\begin{equation}
f_{11}+iA_{11}f=0.  \label{1a}
\end{equation}
It is called \textit{shrinking} if $\mu >0$, \textit{steady} if $\mu =0$,
and \textit{expanding} if $\mu <0$. The equation (\ref{1a}) is corresponding
to the CR vector field condition as in (\ref{3b}) below.
\end{definition}

Next, we relate CR Yamabe solitons and self-similar solutions to the CR Ymamabe flow. 
A special class of solutions to the CR Yamabe flow (\ref{1b}) is given by 
self-similar solutions, whose contact forms $\theta_t$ deform under the CR Yamabe 
flow only by a scaling function depending on $t$ and reparametrizations by a 1-parameter 
family of contact diffeomorphisms; meanwhile, the CR structure $J$ shall be 
invariant under these diffeomorphisms.

Precisely, let $\Phi _{t}:M\rightarrow M$ be a one-parameter family of
contact diffeomorphisms generated by a CR vector field $X_{\widetilde{f}}$ as
above on $M$ with $\Phi _{0}=id_{M}$. Set 
\begin{equation*}
\theta _{t}:=\rho (t)\Phi _{t}^{\ast }\mathring{\theta},\  \  \  \rho (t)>0,\  \
\  \rho (0)=1.
\end{equation*}%
Then $(M, J, \theta _{t})$ is a self-similar solution to the CR Yamabe flow (\ref{1b}) whenever 
\begin{equation*}
-2W\theta _{t}=\frac{\partial }{\partial t}\theta _{t}=\rho ^{\prime
}(t)\Phi _{t}^{\ast }\mathring{\theta}+\rho (t)\Phi _{t}^{\ast }(L_{X_{%
\widetilde{f}}}\mathring{\theta})=(\log \rho )^{\prime }(t)\theta _{t}+L_{X_{%
\widetilde{f}}}\theta _{t}.
\end{equation*}%
It follows from (\ref{41a}) and $d\theta (\mathbf{T},\cdot )=0$ that 
\begin{equation*}
X_{\widetilde{f}}\rfloor \ d\theta _{t}=d\widetilde{f}-\widetilde{f}%
_{0}\theta _{t}
\end{equation*}%
and then 
\begin{equation*}
L_{X\widetilde{_{f}}}\theta _{t}=X_{\widetilde{f}}\rfloor d\theta _{t}+d(X_{%
\widetilde{f}}\rfloor \theta _{t})=X_{f}\rfloor \ d\theta _{t}-d\widetilde{f}%
=-\widetilde{f}_{0}\theta _{t}.
\end{equation*}%
Hence, 
\begin{equation*}
-2W\theta _{t}-(\log \rho )^{\prime }(t)\theta _{t}=L_{X_{\widetilde{f}%
}}\theta _{t}=-\widetilde{f}_{0}\theta _{t}.
\end{equation*}%
In short, every self-similar solution satisfies the equation %
\begin{equation*}
2W\theta _{t}+L_{X_{\widetilde{f}}}\theta _{t}=2W\theta _{t}-\widetilde{f}%
_{0}\theta _{t}=-(\log \rho )^{\prime }(t)\theta _{t}
\end{equation*}%
for some scaling function $\rho (t)>0$ and potential function $\widetilde{f}%
:M\rightarrow \mathbb{R}$. When $t=0$, we obtain the equation  
\begin{equation*}
W-\frac{1}{2}\widetilde{f}_{0}=-\frac{1}{2}\rho ^{\prime }(0)
\end{equation*}%
for $(M, J, \theta)$. 
If we put $f=-\widetilde{f}$ and denote $\mu= -\frac{1}{2}\rho ^{\prime }(0)$ then we have 
\begin{equation*}
W+\frac{1}{2}f_{0}=\mu.
\end{equation*}%
Recall that we still need to make sure $J$ is invariant under the contact diffeomorphism,
i.e.,
\begin{equation}
0=L_{X_{f}}J \equiv 2(f_{11}+iA_{11}f)\theta ^{1}\otimes Z_{\bar{1}
}+2(f_{\bar{1}\bar{1}}-iA_{\bar{1}\bar{1}}f)\theta ^{%
\bar{1}}\otimes Z_{1}\  \mathrm{mod }\ \theta .  \label{3b}
\end{equation}
Thus $f_{11}+iA_{11}f=0$. 
(We refer to \cite{cl}, p. 240, for a proof of the expression of $L_{X_{f}}J$.)

Conversely, starting from a CR Yamabe soliton $(M, J, \mathring{\theta}, f, \mu)$ which
satisfies $W+\frac{1}{2}f_{0}=\mu $ and $f_{11}+iA_{11}f=0$, one can solve equations 
\begin{equation*}
\left \{ 
\begin{array}{rl}
\frac{\partial }{\partial t}\rho (t)|_{t=0} & =-2\mu , \\ 
\rho (0) & =1,%
\end{array}%
\right.
\end{equation*}%
and for $f(t):=f\circ \Phi _{t}$ 
\begin{equation*}
\left \{ 
\begin{array}{rl}
\frac{\partial }{\partial t}\Phi _{t} & =X_{f(t)}, \\ 
\Phi _{0} & =id_{M},%
\end{array}%
\right.
\end{equation*}%
to get $\rho (t)$ and $\Phi _{t}$, and thus a self-similar solution $\theta _{t}=\rho (t)\Phi _{t}^{\ast }\mathring{\theta}$ 
of the CR Yamabe flow. Note
that it suffices to take $\rho (t)=1-2\mu t$. This would be the unique solution once we
establish the uniqueness for CR Yamabe flow. The same phenomenon occurred in
the theory of Ricci solitons (cf. \cite{cln}).

In this article, we focus on CR Yamabe solitons instead of their corresponding self-similar solutions. 

To prove Theorem 1.1, we shall need a certain differential Harnack quantity for the three-dimensional CR Yamabe flow.
We remark that in the Ricci flow, Hamilton found a conserved quantity which vanishes identically for expanding Ricci
solitons and showed that such a quantity is nonnegative for generic solutions with
positive curvature operator. This quantity is called the differential Harnack quantity, or Li-Yau-Hamilton 
(LYH) quantity. In the paper \cite{cc}, J.-H. Cheng and the second named author have indicated a similar Harnack quantity
for the CR Yamabe flow (\ref{1b}). For the reader's convenience, we give a detailed proof here.

\begin{lemma} A three-dimensional CR Yamabe soliton satisfies: 
\begin{equation}
4\Delta _{b}W+2W(W-\mu )-W_{0}f-\langle \nabla _{b}W,J(\nabla _{b}f)\rangle_{\theta }=0,  \label{3}
\end{equation}%
\end{lemma}

\begin{proof}
Recall that $\Delta _{b}W=W_{1\bar{1}}+W_{\bar{1}1}$. We first
differentiate the soliton equation (\ref{1}) and obtain
\begin{equation*}
W_{1\bar{1}}=-\frac{1}{2}f_{01\bar{1}}=\frac{i}{2}(f_{1\bar{1}}-f_{\bar{1}%
1})_{1\bar{1}}.
\end{equation*}%
The two terms appear on the right hand side are higher derivatives of $f$ 
which can be reduced by using the CR vector field condition%
\begin{equation*}
f_{11}+iA_{11}f=0.
\end{equation*}
Indeed, by using commutation relations (\ref{2010a}), one derives%
\begin{equation*}
\begin{array}{ccl}
f_{1\bar{1}1} & = & f_{11\bar{1}}-if_{10}-f_{1}W \\ 
& = & -i(A_{11}f)_{\bar{1}}-if_{01}+iA_{11}f_{\bar{1}}-f_{1}W \\ 
& = & -iA_{11,\bar{1}}f+2iW_{1}-f_{1}W.%
\end{array}%
\end{equation*}%
Differentiate it in the direction $Z_{\bar{1}}$, one achieves 
\begin{equation*}
f_{1\bar{1}1\bar{1}}=-iA_{11,\bar{1}\bar{1}}f-iA_{11,\bar{1}}f_{\bar{1}%
}+2iW_{1\bar{1}}-f_{1\bar{1}}W-f_{1}W_{\bar{1}}.
\end{equation*}%
On the other hand, the differentiation of its conjugation in the $Z_{1}$
direction gives an expression of $f_{\bar{1}1\bar{1}1}$. After changing the $%
3$rd and $4$th indices, one obtains 
\begin{equation*}
\begin{array}{ccl}
f_{\bar{1}11\bar{1}}
&=& f_{\bar{1}1\bar{1}1}+if_{\bar{1}10} \\
&=& iA_{\bar{1}\bar{1},11}f+iA_{\bar{1}\bar{1},1}f_{1}
    -2iW_{\bar{1}1}-f_{\bar{1}1}W-f_{\bar{1}}W_{1}+if_{\bar{1}10},
\end{array}
\end{equation*}
where we can further reduce the bad term $f_{\bar{1}10}$ by noticing that
\begin{equation*}
\begin{array}{ccl}
f_{\bar{1}10} & = & f_{0\bar{1}1}-A_{11}f_{\bar{1}\bar{1}}-A_{11,\bar{1}}f_{%
\bar{1}}-A_{\bar{1}\bar{1}}f_{11}-A_{\bar{1}\bar{1},1}f_{1} \\ 
& = & -2W_{\bar{1}1}-A_{11}f_{\bar{1}\bar{1}}-A_{11,\bar{1}}f_{\bar{1}}-A_{%
\bar{1}\bar{1}}f_{11}-A_{\bar{1}\bar{1},1}f_{1} \\ 
& = & -2W_{\bar{1}1}-A_{11,\bar{1}}f_{\bar{1}}-A_{\bar{1}\bar{1},1}f_{1}.%
\end{array}%
\end{equation*}%

Substituting our expressions for $f_{1\bar{1}1\bar{1}}$ and $f_{\bar{1}11%
\bar{1}}$ into the equation for $W_{1\bar{1}}$, we get%
\begin{equation*}
\begin{array}{ccl}
2W_{1\bar{1}} & = & A_{11,\bar{1}\bar{1}}f+A_{11,\bar{1}}f_{\bar{1}}-2W_{1%
\bar{1}}-if_{1\bar{1}}W-if_{1}W_{\bar{1}} \\ 
&  & +A_{\bar{1}\bar{1},11}f+A_{\bar{1}\bar{1},1}f_{1}-2W_{\bar{1}1}+if_{%
\bar{1}1}W+if_{\bar{1}}W_{1} \\ 
&  & -2W_{\bar{1}1}-A_{11,\bar{1}}f_{\bar{1}}-A_{\bar{1}\bar{1},1}f_{1}.%
\end{array}%
\end{equation*}%

Finally, by the CR Bianchi identity $A_{11,\bar{1}\bar{1}}+A_{\bar{1}\bar{1}%
,11}=W_{0} $, we have the Harnack quantity%
\begin{equation}
\begin{array}{ccl}
4\Delta _{b}W & = & (A_{11,\bar{1}\bar{1}}+A_{\bar{1}\bar{1},11})f-i(f_{1%
\bar{1}}-f_{\bar{1}1})W-i(f_{1}W_{\bar{1}}-f_{\bar{1}}W_{1}) \\ 
& = & W_{0}f-i(if_{0})W-i(f_{1}W_{\bar{1}}-f_{\bar{1}}W_{1}) \\ 
& = & -2W(W-\mu )+W_{0}f+\langle \nabla _{b}W,J(\nabla _{b}f)\rangle
_{\theta }.%
\end{array}
\label{3a}
\end{equation}%
This finishes the proof of (\ref{3}). 

\end{proof}

Now by using the Harnack quantity (\ref{3}), we can show that every closed
CR Yamabe soliton has constant Tanaka-Webster curvature. This computation is similar to the one given by S.-Y. Hsu \cite{hsu} for closed Riemannian 
Yamabe solitons.  \\

\textbf{The Proof of Theorem \ref{A}:}

\begin{proof}
Integrating the Harnack quantity (\ref{3}), 
one derives that 
\begin{equation*}
\begin{array}{ccl}
0 & = & \int_{M}[-2W(W-\mu )+W_{0}f-i(f_{1}W_{\bar{1}}-f_{\bar{1}}W_{1})]d\mu
\\ 
& = & -2\int_{M}W(W-\mu )d\mu -\int_{M}Wf_{0}d\mu +\int_{M}i(f_{1\bar{1}}-f_{%
\bar{1}1})Wd\mu \\ 
& = & -2\int_{M}W(W-\mu )d\mu -2\int_{M}Wf_{0}d\mu \\ 
& = & -2\int_{M}W(W-\mu )d\mu +4\int_{M}W(W-\mu )d\mu \\ 
& = & 2\int_{M}W(W-\mu )d\mu .%
\end{array}%
\end{equation*}%
Together with the fact that $\int_{M}(W-\mu )d\mu =-\frac{1}{2}\int_{M}f_{0}d\mu 
=\frac{1}{2}\int_{M}f \cdot \mathrm{div}\mathbf{T}\ d\mu =0$, we obtain 
\begin{equation*}
\int_{M}(W-\mu )^{2}d\mu =\int_{M}W(W-\mu )d\mu -\mu \int_{M}(W-\mu )d\mu =0.
\end{equation*}
\end{proof}

Next, to prepare the proof of Theorem 1.2, we first recall the definitions of CR Paneitz operator and the CR $Q$%
-curvature:

\begin{definition}
Let $(M,J,\theta)$ be a closed $3$-dimensional pseudohermitian manifold.
We define, as in \cite{l1}, that  
\begin{equation}
P\varphi: =(P_{1}\varphi )\theta ^{1} :=(\varphi _{\bar{1}}{}^{\bar{1}}{}_{1}+iA_{11}\varphi ^{1})\theta
^{1}  \label{4}
\end{equation}
and 
\begin{equation*}
\overline{P}\varphi :=(\overline{P}_{1})\theta ^{\bar{1}}.
\end{equation*}
The operator $P$ characterizes CR-pluriharmonic functions. The CR Paneitz operator $P_{0}$ is defined by 
\begin{equation}
P_{0}\varphi =\left( \delta _{b}(P\varphi )+\overline{\delta }_{b}(\overline{%
P}\varphi )\right),  \label{4a}
\end{equation}%
where $\delta _{b}$ is the divergence operator that takes $(1,0)$-forms to
functions by $\delta _{b}(\sigma _{1}\theta ^{1})=\sigma _{1,}{}^{1}$, and
similarly, $\bar{\delta}_{b}(\sigma _{\bar{1}}\theta ^{\bar{1}})=\sigma _{%
\bar{1},}{}^{\bar{1}}$. \ One can check that $P_{0}$ is self-adjoint, that
is, $\left \langle P_{0}\varphi ,\psi \right \rangle =\left \langle \varphi
,P_{0}\psi \right \rangle $ for all smooth functions $\varphi $ and $\psi $.
For the details about these operators, the reader can consult 
\cite{gl}, \cite{hi}, \cite{l1}, and \cite{cchi}.
\end{definition}

\begin{definition}
(\cite{hi, ccc}) Define the tensor 
\begin{equation*}
R_{1}\theta ^{1}:=(W_{1}-iA_{11,\bar{1}})\theta ^{1}.
\end{equation*}%
Use this to define the CR $Q$-curvature on\textbf{\ }$\left( M^{3},\text{ }%
\theta \right) $ by
\begin{equation}
Q:=-cR_{1,\bar{1}}=-\frac{c}{2}(\Delta _{b}W+2\mathrm{Im}A_{11,\bar{1%
}\bar{1}}),  \label{33b}
\end{equation}%
for some constant $c$. The last equality follows from \cite[Lemma 5.4]{hi}.
Clearly, vanishing of $R_{1}$ guarantees vanishing of the CR $Q$-curvature.
\end{definition}

\begin{lemma}
\label{l42} (\cite[Lemma 5.4]{hi}) Let $(M^{3},J,\theta )$ be a
pseudohermitian manifold. Then, for rescaled contact form $\widetilde{\theta 
}=e^{2g}\theta ,$ we have%
\begin{equation}
\widetilde{R}_{1}=e^{-3g}(R_{1}-6P_{1}g)  \label{33}
\end{equation}%
and thus 
\begin{equation}
\widetilde{R}_{1,\bar{1}}=e^{-4g}(R_{1,\bar{1}}-6C_{\theta }g),\  \
C_{\theta }g=(P_{1}g)_{\bar{1}}.  \label{33a}
\end{equation}
\end{lemma}

\begin{corollary}
\label{c42} Let $(M^{3},J,\theta )$ be a pseudohermitian manifold. 
If $\widetilde{\theta }=e^{2g}\theta $ has vanishing CR 
$\widetilde{Q}$-curvature for some function $g$, then
\begin{equation*}
\Delta _{b}W+2\mathrm{Im}A_{11,\bar{1}\bar{1}}=12C_{\theta }g.
\end{equation*}
\end{corollary}

\begin{proof}
The result follows from (\ref{33a}) and (\ref{33b}) easily as 
\begin{equation*}
0=R_{1,\bar{1}}-6C_{\theta }g=\frac{1}{2}(\Delta _{b}W+2\mathrm{Im}A_{11,%
\bar{1}\bar{1}})-6C_{\theta }g.
\end{equation*}
\end{proof}

\textbf{The Proof of Theorem \ref{B} : }

\begin{proof}
It follows from the CR vector field condition (\ref{1a}) that 
\begin{equation*}
0=\int_{M}(A_{\bar{1}\bar{1}}f_{11}+i|A_{11}|^{2}f)d\mu
=\int_{M}(A_{\bar{1}\bar{1},11}f)d\mu +i\int_{M}|A_{11}|^{2}fd\mu
\end{equation*}%
and 
\begin{equation*}
0=\int_{M}(A_{11}f_{\bar{1}\bar{1}}-i|A_{11}|^{2}f)d\mu
=\int_{M}(A_{11,\bar{1}\bar{1}}f)d\mu -i\int_{M}|A_{11}|^{2}fd\mu. 
\end{equation*}%
Thus, 
\begin{equation}
0=\int_{M}(\mathrm{Im}A_{11,\bar{1}\bar{1}}f)d\mu
-\int_{M}|A_{11}|^{2}fd\mu .  \label{2014}
\end{equation}%
Now since the closed CR Yamabe soliton has constant Tanaka-Webster scalar
curvature, from Corollary \ref{c42} we have 
\begin{equation*}
\mathrm{Im}A_{11,\bar{1}\bar{1}}=6C_{\theta }g.
\end{equation*}%
This together with (\ref{2014}) implies that %
\begin{equation*}
\int_{M}|A_{11}|^{2}fd\mu =\int_{M}(\mathrm{Im}A_{11,\bar{1}\bar{1}%
}f)d\mu =6\int_{M}(C_{\theta }g)fd\mu =3\int_{M}(P_{0}g)fd\mu
=3\int_{M}(P_{0}f)gd\mu =0.
\end{equation*}%
Here $P_{0}g=C_{\theta }g+\overline{C_{\theta }g}=2C_{\theta }g.$ 
Since $f$ vanishes nowhere, we obtain 
\begin{equation*}
A_{11}=0.
\end{equation*}
\end{proof}

\textbf{The Proof of Corollary \ref{C1} :}

\begin{proof}
In the paper \cite{caoc}, J. Cao and the second named author proved that
if $(M,J,\theta )$ is the smooth boundary of a bounded strictly pseudoconvex
domain $\Omega $ in a complete Stein manifold $V^{2}$, then there exists a
smooth function $g$ such that $\widetilde{\theta }=e^{2g}\theta $ has
vanishing CR $\widetilde{Q}$-curvature. 
Hence, Corollary \ref{C1} follows from Theorem \ref{B}.
\end{proof}

\section{Complete $3$-dimensional Pseudo-Gradient CR Yamabe Solitons}

We first recall that the family of Webster adapted metrics $h^{\lambda }$ of 
$(M,J,\theta )$ is given as the following: 
\begin{equation*}
h^{\lambda }=\frac{1}{2}h+\lambda ^{-2}\theta ^{2},\  \lambda >0, 
\end{equation*}%
with 
\begin{equation*}
h(X,Y)=d\theta (X,JY).
\end{equation*}%

In this section, we discuss basic properties of complete CR 
Yamabe pseudo-gradient solitons, which are analogue to
gradient Ricci solitons. Using these properties, 
we obtain the structure theorem for nontrivial complete $3$-dimensional 
pseudo-gradient CR Yamabe solitons in the next section.

\begin{definition}
A complete pseudohermitian $3$-manifold $(M^{3},J,$ $\theta )$ 
is called a pseudo-gradient 
CR Yamabe soliton if there exist a smooth function $\varphi $ and $\mu\in {\mathbb R}$ such that
\begin{equation}
W+\frac{1}{2}\Delta _{b}\varphi =\mu,  \label{2}
\end{equation}%
with 
\begin{equation}
\varphi _{11}=0\  \  \  \  \mathrm{and}\  \  \  \  \varphi _{0}=0.  \label{2b}
\end{equation}
$(M^{3},J,$ $\theta )$ is said to be {\it trivial} if $\varphi$ is a constant function.
\end{definition}

\begin{remark}
(i) By Theorem \ref{A}, it is easy to see that a closed CR Yamabe soliton is always a trivial pseudo-gradient soliton. 
However, given a non-compact CR Yamabe soliton $(M,J,\theta,f,\mu)$, there
might not even exist a function $\varphi$ such that $\Delta _{b}\varphi =f_{0}$.

(ii) The constraint $\varphi_{0}=0$ arises naturally when 
one requires the following equivalent expression of (\ref{2}) to hold (cf. the proof of Lemma 4.1 (i)):
\begin{equation*}
R_{1\bar{1}}+\varphi _{1\bar{1}}=\mu h_{1\bar{1}}.
\end{equation*}%
Here the positive function $h_{1\bar{1}}$ is the Levi metric with respect
to the Levi form $d\theta =ih_{1\bar{1}}\theta ^{1}\wedge \theta ^{\bar{1}}$
and $R_{1\bar{1}}=Wh_{1\bar{1}}.$
This expression is analogous to the gradient K\"ahler-Ricci soliton where the constraint
$\varphi_{11}=0$ is parallel to the infinitesimal automorphism condition in that case 
(cf. page 93 in \cite{ccgg+}). 
This justifies the term ``pseudo-gradient" in our definition in consideration 
of the theory of gradient Ricci soliton.

(iii) Gradient Ricci solitons can be viewed as generalized Einstein manifolds in 
the weighted geometry sense and analogue theory was built for pseudo-gradient CR 
Yamabe solitons in weighted pseudohermitian geometry \cite{ckl}.
In particular, a pseudo-gradient CR Yamabe soliton has constant Bakry-\'Emery 
pseudohermitian Ricci curvature
$$Ric(\Delta _{\varphi })(X,X):=R_{1\bar{1}}X^{1}X^{\bar{1}}+\mathrm{Re}[\varphi _{1\bar{1}}X^{1}X^{\bar{1}}]=\mu |X|^{2}$$
and the Bakry-\'Emery pseudohermitian torsion of it is the same as the pseudohermitian torsion: 
$$Tor(\Delta _{\varphi })(X,X):=2\mathrm{Re}[(iA_{\bar{1}\bar{1}}+\varphi _{\bar{1}\bar{1}})X_{1}X_{1}]=Tor(X,X),$$
where $X=X^{1}Z_{1}+X^{\bar{1}}Z_{\bar{1}}\in T_{1,0}\oplus T_{0,1}$. 

\end{remark}

Similar to the case of complete gradient Ricci solitons, we have

\begin{lemma}
\label{l51} Let  $\left( M^{3},J, \theta , \varphi ,\mu \right)$ be a  
complete three-dimensional pseudo-gradient CR Yamabe soliton of vanishing torsion, then

(i) 
\begin{equation*}
W+\frac{1}{2}|\nabla _{b}\varphi |^{2}-\mu \varphi =C
\end{equation*}%
for some constant $C$;

(ii) 
\begin{equation*}
\nabla _{b}(We^{-\varphi })=0.
\end{equation*}
\end{lemma}

\begin{proof}
(i) Since $0=\varphi _{0}=-i(\varphi _{1\bar{1}}-\varphi _{\bar{1}1})$, we have
\begin{equation*}
\Delta _{b}\varphi =2\varphi _{1\bar{1}}.
\end{equation*}%
It follows from (\ref{2}) that 
\begin{equation}
W_{1}+\varphi _{1\bar{1}1}=0.  \label{41}
\end{equation}%

By the commutation relation (\ref{2010a}) , $\varphi _{0}=0$ and $\varphi _{11}=0$,%
\begin{equation}
\begin{array}{ccl}
\varphi _{1\bar{1}1} & = & \varphi _{11\bar{1}}-i\varphi
_{10}-W\varphi _{1} \\ 
& = & i\varphi _{01}+iA_{11}\varphi _{\bar{1}}-W\varphi _{1} \\ 
& = & iA_{11}\varphi _{\bar{1}}-W\varphi _{1}.%
\end{array}
\label{42}
\end{equation}%
Hence, from (\ref{41}) and (\ref{42}),  we know that
\begin{equation}
W_{1}=-iA_{11}\varphi _{\bar{1}}+W\varphi _{1}.  \label{43}
\end{equation}
It follows from (\ref{43}) and $\varphi _{11}=0$ that 
\begin{equation*}
\begin{array}{ccl}
(W+\frac{1}{2}|\nabla _{b}\varphi |^{2}-\mu \varphi )_{1} & = & 
W_{1}+(\varphi _{1}\varphi _{\bar{1}})_{1}-\mu \varphi _{1} \\ 
& = & -iA_{11}\varphi _{\bar{1}}+(W-\mu )\varphi _{1}+\varphi
_{1}\varphi _{\bar{1}}{}_{1} \\ 
& = & -iA_{11}\varphi _{\bar{1}}-\varphi _{\bar{1}}{}_{1}\varphi
_{1}+\varphi _{1}\varphi _{\bar{1}}{}_{1} \\ 
& = & -iA_{11}\varphi _{\bar{1}}.%
\end{array}%
\end{equation*}%
Hence,  
\begin{equation*}
(W+\frac{1}{2}|\nabla _{b}\varphi |^{2}-\mu \varphi )_{1}=0
\end{equation*}%
and 
\begin{equation*}
\frac{1}{2}|\nabla _{b}\varphi |^{2}+W-\mu \varphi =C
\end{equation*}%
for some constant $C$ if $A_{11}=0$.

(ii) If $A_{11}=0,$ then 
\begin{equation*}
\begin{array}{ccl}
(We^{-\varphi })_{1} & = & (W_{1}-W\varphi _{1})e^{-\varphi } \\ 
& = & (-iA_{11}\varphi _{\bar{1}}+W\varphi _{1}-W\varphi
_{1})e^{-\varphi } \\ 
& = & -iA_{11}\varphi _{\bar{1}}e^{-\varphi } \\ 
& = & 0.%
\end{array}%
\end{equation*}
\end{proof}

We rewrite the real version of complete pseudo-gradient CR Yamabe soliton as 
follows:

\begin{lemma}
\label{l52} Let $(M^{3},J,\theta )$ be a complete three-dimensional pseudohermitian manifold. Then $\left( M^{3},J,\theta , \varphi , \mu \right) $ is a
complete pseudo-gradient CR Yamabe soliton if and only if the following identities hold: 

\medskip
(i) 
\begin{equation*}
\varphi _{e_{1}e_{1}}=\varphi _{e_{2}e_{2}} \quad \mbox{and} \quad 
\varphi _{e_{1}e_{2}}=\varphi _{e_{2}e_{1}}=0;
\end{equation*}

(ii) 
\begin{equation*}
W+\frac{1}{2}\varphi _{e_{1}e_{1}}=W+\frac{1}{2}\varphi _{e_{2}e_{2}}=\mu .
\end{equation*}
\end{lemma}

\begin{proof}
(i) From the definitions of $Z_{1}=\frac{1}{2}(e_{1}-ie_{2})$ and $\varphi
_{e_{i}e_{j}}=e_{j}e_{i}\varphi -\sigma _{i}^{k}(e_{j})\varphi _{e_{k}},$ we
have 
\begin{equation*}
\varphi _{1}=\frac{1}{2}(\varphi _{e_{1}}-i\varphi _{e_{2}})
\end{equation*}%
and 
\begin{equation}
\begin{array}{ccl}
4\varphi _{11} & = & 4[Z_{1}\varphi _{1}-\theta _{1}{}^{1}(Z_{1})\varphi
_{1}] \\ 
& = & 4[Z_{1}\varphi _{1}-i\sigma _{1}^{2}(Z_{1})\varphi _{1}] \\ 
& = & (e_{1}e_{1}\varphi -\sigma _{1}^{2}(e_{1})\varphi
_{e_{2}})-(e_{2}e_{2}\varphi -\sigma _{2}^{1}(e_{2})\varphi _{e_{1}}) \\ 
&  & -i(e_{1}e_{2}\varphi -\sigma _{2}^{1}(e_{1})\varphi
_{e_{1}})-i(e_{2}e_{1}\varphi -\sigma _{1}^{2}(e_{2})\varphi _{e_{2}}) \\ 
& = & (\varphi _{e_{1}e_{1}}-\varphi _{e_{2}e_{2}})-i(\varphi
_{e_{1}e_{2}}+\varphi _{e_{2}e_{1}}).%
\end{array}
\label{2014c}
\end{equation}%

On the other hand, from (\ref{2010b}) we see that $\varphi _{0}=0$ if and only if 
\begin{equation*}
\varphi _{e_{1}e_{2}}=\varphi _{e_{2}e_{1}}.
\end{equation*}%
Thus, we conclude that $\varphi _{11}=0$ and $\varphi _{0}=0$ if and only if 
\begin{equation*}
\varphi _{e_{1}e_{1}}=\varphi _{e_{2}e_{2}} \quad \mbox{and} \quad 
\varphi _{e_{1}e_{2}}=\varphi _{e_{2}e_{1}}=0.
\end{equation*}

(ii) follows from Eq. (4.1) in the Definition 4.1, (\ref{2014d}),  and part (i) proved above. 

\end{proof}

\section{The Structure of Pseudo-Gradient  CR Yamabe Solitons}

In this section, by using ideas of the recent work of the first-named author with X. Sun and 
Y. Zhang \cite{csz}  on the structure of Yamabe gradient solitons, we obtain
the structure of pseudo-gradient CR Yamabe solitons of vanishing torsion.

Let $\{\omega^{1}:=\mathrm{Re}(\theta ^{1}),
\omega^{2}:=\mathrm{Im}(\theta ^{1}),
\omega^3:=\lambda^{-1}\theta\}$ be the orthonormal coframe dual to $\{e_{1},e_{2},e_{3}=\lambda\mathbf{T}\}$ with respect to the Webster adapted metric 
$h^{\lambda }$ defined on $(M,J,\theta )$ as in Section 4.
Then we have the following Riemannian structure equations: 

\begin{equation}
\begin{split}
d\omega ^{\alpha }& =\omega ^{\beta }\wedge \omega _{\beta }^{\alpha },\quad
1\leqslant \alpha ,\beta \leqslant 3, \\
0& =\omega _{\alpha }^{\beta }+\omega _{\beta }^{\alpha }, \\
d\omega _{\alpha }^{\beta }& =\omega _{\alpha }^{\gamma }\wedge \omega
_{\gamma }^{\beta }+\frac{1}{2}R_{\alpha \beta \gamma \delta }^{\lambda
}\omega ^{\gamma }\wedge \omega ^{\delta },\quad 1\leqslant \alpha ,\beta
,\gamma ,\delta \leqslant 3,
\end{split}
\label{strueq3}
\end{equation}%
and $\theta ^{1}=\omega ^{1}+i\omega ^{2}$ satisfies the structure
equations as in (\ref{100240}) and (\ref{100241}). Here $R_{\alpha \beta
\gamma \delta }^{\lambda }$ denotes the Riemannian curvature tensor.

Moreover, from \cite[(3.6) and (3.7)]{cchi} we have 
\begin{equation}
\theta _{1}{}^{1}=i(\omega _{1}^{2}-\lambda ^{-2}\theta )  \label{2012a}
\end{equation}%
and%
\begin{equation}
\begin{array}{ccc}
\omega _{1}^{3} & = & \left( -\lambda \mathrm{Re}A_{\bar{1}\bar{1}}\right)
\omega ^{1}+\left( -\lambda \mathrm{Im}A_{\bar{1}\bar{1}}+\lambda ^{-1}\right)
\omega ^{2}, \\ 
\omega _{2}^{3} & = & \left( -\lambda \  \mathrm{Im}\ A_{\bar{1}\bar{1}%
}-\lambda ^{-1}\right) \omega ^{1}+\left( \lambda \mathrm{Re}A_{\bar{1}\bar{1}%
}\right) \omega ^{2}.%
\end{array}
\label{2012b}
\end{equation}

Now we state the relation between the Ricci tensor $%
R_{ij}^{\lambda }$ of the Webster metric $h^{\lambda }$ and the
pseudohermitian Ricci tensor $R_{1\bar{1}}=Wh_{1\bar{1}}$. 

\begin{lemma}
\label{l53} (\cite[Theorem 3.1]{cchi}) Let $(M,J,\theta )$ be a complete
pseudohermitian $3$-manifold and $R_{\alpha \beta }^{\lambda }$ be the Ricci
curvature tensor with respect to the Webster metric $h^{\lambda }$
for some positive number $\lambda>0$. 
If the pseudohermitian torsion $A_{11}$ vanishes, then 
\begin{equation}
\  \left( R_{\alpha \beta }^{\lambda }\right) =%
\begin{pmatrix}
2W-2\lambda ^{-2} & 0 & 0 \\ 
0 & 2W-2\lambda ^{-2} & 0 \\ 
0 & 0 & 2\lambda ^{-2}%
\end{pmatrix}%
. \label{2014e}
\end{equation}%
Here $W$ is the Tanaka-Webster scalar curvature. In particular, 
the scalar curvature 
\begin{equation}
R^{\lambda }=4W-2\lambda ^{-2}.  \label{2014f}
\end{equation}
\end{lemma}

Now we are ready to prove the classification theorem of complete pseudo-gradient CR Yamabe solitons.

\begin{lemma}
\label{const} Let $\left( M^{3},J,\theta ,\varphi, \mu \right) $
be a complete pseudo-gradient CR Yamabe soliton. Then $|\nabla _{b}\varphi |$ is
constant on each level set of the potential function $\varphi $.
\end{lemma}

\begin{proof}
One computes that 
\begin{align*}
\nabla(|\nabla_b\varphi |^2) = & (\varphi_1\varphi_{\bar{1}})_1Z_{%
\bar{1}} +(\varphi_1\varphi_{\bar{1}})_{\bar{1}}Z_1\\
= & (\mu-W)h_{\bar{1}1}\varphi_1Z_{\bar{1}} + (\mu-W)h_{1\bar{1}%
}\varphi_{\bar{1}}Z_1 \\
=& (\mu -W)\nabla \varphi.
\end{align*}
So the level sets of $|\nabla_b\varphi|$ coincide with the level sets of $%
\varphi$.
\end{proof}

Combining the lemma above with the soliton equation $W+\frac{1}{2}\Delta_b\varphi=\mu$ 
and the conserved quantity $W+\frac{1}{2}|\nabla_b\varphi|^2-\mu\varphi=C$, we immediately obtain the following

\begin{lemma}
\label{isop} Let $\left( M^{3},J,\theta ,\varphi, \mu \right) $
be a complete pseudo-gradient CR Yamabe soliton. Then, the potential function $\varphi$ is an isoparametric function on the Riemannain 3-manifold $(M,h^\lambda)$. 
\end{lemma}

\begin{remark}\label{isopara}
In \'E. Cartan's study of Lie groups, a function $f$ is called an 
isoparametric function if both $|\nabla f|$ and $\Delta f$ depend 
only on the value of $f$ (i.e., constant on each level set of $f$). 
Note that, by the work of Q. M. Wang in \cite{wang} on isoparametric functions, the only possible singular level sets of the potential function $\varphi$ are smooth submanifolds 
$$\Sigma_+:= \{\varphi = \max_{x\in M}\varphi(x)\}\ \mbox{ and }\ \Sigma_-:= \{\varphi = \min_{x\in M}\varphi(x)\},$$
which are called  the focal varieties of $\varphi$.
Moreover, regular level sets are ``tubes" over either $\Sigma_+$ or $\Sigma_{-}$ and these ``tube" leaves together form foliations of tubular neighborhoods around either $\Sigma_+$ or $\Sigma_-$. 
Finally, the argument in \cite{wang} can be further used to show that if an isoparametric function $f$ is locally constant then $f$ must be a constant function. Hence, the critical set of our non-trivial potential function $\varphi$ is at most $2$-dimensional. The relevant details will be provided in the Appendix for the reader's convenience.
\end{remark}

\textbf{The Proof of Theorem \ref{E} :}

\begin{proof} In view of Lemma \ref {isop}, it remains to classify the possible diffeomorphism types of $M^3$. 
Consider again the adapted metric  
\begin{equation*}
g_{\lambda }:=\frac{1}{2}d\theta (\cdot ,J\cdot )+\lambda ^{-2}\theta ^{2}.
\end{equation*}%
It follows from Lemma \ref{l53} that, when torsion vanishes, 
\begin{equation*}
Ric^{\lambda}=\left( 
\begin{array}{ccc}
2W-2\lambda ^{-2} & 0 & 0 \\ 
0 & 2W-2\lambda ^{-2} & 0 \\ 
0 & 0 & 2\lambda ^{-2}%
\end{array}%
\right)
\end{equation*}%
with respect to the orthonormal frame $\{e_{1},e_{2},e_{3}=\lambda \mathbf{T}\}$, where $%
e_{1}=Z_{1}+Z_{\bar{1}}$ and $e_{2}=i(Z_{1}-Z_{\bar{1}})$. 

When dimension $n=3$, 
\begin{equation*}
R^{\lambda}_{2323}-R^{\lambda}_{1313}=R^{\lambda}_{22}-R^{\lambda}_{11}=0
\end{equation*}
and 
\begin{equation*}
R^{\lambda}_{2323}+R^{\lambda}_{1313}=R^{\lambda}_{33}=2\lambda ^{-2}.
\end{equation*}
So we know that 
\begin{equation*}
R^{\lambda}_{2323}=R^{\lambda}_{1313}=\lambda ^{-2}.
\end{equation*}
Moreover, let $V$ be any unit vector which is a linear combination of $e_{1}$ and $e_{2}$, then
the sectional curvature of the two-plane spanned by $e_{3}$ and $V$ is
merely a linear combination of $R^{\lambda}_{2323}$ and $R^{\lambda}_{1313}$, hence 
\begin{equation*}
Rm^{\lambda}(V,e_{3},V,e_{3})=\lambda ^{-2}.
\end{equation*}%
Note that this property holds only for the $3$-dimensional case, because the
Weyl tensor always vanishes and the curvature operator is diagonalized as
long as $Ric^{\lambda}$ and $g_{\lambda}$ are diagonalized.

Since by assumption the pseudo-gradient CR Yamabe soliton $M^3$ is non-trivial, $\varphi $ is not a constant function. By Lemma \ref{const}, if 
$\nabla \varphi \neq 0$ at some point $p$, then $\nabla \varphi \neq 0$ on
the whole level set $\Sigma _{\varphi (p)}:=\{ \varphi =\varphi (p)\}$. Such 
$\Sigma _{\varphi (p)}$ (with non-vanishing $\nabla \varphi $) are called
regular. We are going to compute the Gaussian curvature of regular level
surfaces.

\smallskip
{\bf Claim}:  Each regular level surface $\Sigma$ of $\varphi $ has zero Gaussian curvature. 
\smallskip
  
Indeed, for any regular level surface $\Sigma$ of $\varphi$, first note that $\varphi _{0}=0$ means that $\mathbf{T}$ lies in the tangent subbundle $T\Sigma $ 
and is perpendicular to $\nabla \varphi $, so one can choose 
$$E_{1}=\frac{\nabla \varphi}{|\nabla\varphi|}=\alpha e_1 +\beta e_2\ \mbox{ and }\
E_{2}=\beta e_1 -\alpha e_2,$$
where $\alpha, \beta $ are real-valued functions on $M$ 
such that $E_1,E_2$ and $E_{3}=e_3=\lambda\mathbf{T}$ form an orthonormal frame. 
The Gaussian curvature $K^{\lambda}$ of $\Sigma$  
is given by the Gauss equation 
\begin{equation*}
K^{\lambda}=Rm^{\lambda}(E_2,E_3,E_2,E_3)
-\mathrm{II}(E_{2},E_{3})^{2}+\mathrm{II}(E_{2},E_{2})\mathrm{II}(E_{3},E_{3}),
\end{equation*}%
where $\mathrm{II}$, the second fundamental form, can be computed
via the connection $1$-forms (cf. equation (\ref{2012b}))
$$ \omega_1^3=\lambda^{-1}\omega^2\ \mbox{ and }\ \omega_2^3= -\lambda^{-1}\omega^1.$$
Indeed, we have
$$\mathrm{II}(E_{3},E_{3})
=\left \langle \nabla _{e_{3}}e_{3},E_1\right \rangle 
=\omega_3^1(e_3)
=0.$$
Similarly, since
$$\nabla_{E_2} \mathbf{T}
=\lambda^{-1}\nabla_{\beta e_1-\alpha e_2} e_3
=\lambda^{-1}(\beta \omega_3^k(e_1)e_k -\alpha\omega_3^k(e_2)e_k)
=\lambda^{-1}(\beta \lambda^{-1}e_2 +\alpha\lambda^{-1} e_1)
=\lambda^{-2}E_1,$$
we obtain
\begin{align*}
\mathrm{II}(E_{2},E_{3})
&=\left \langle \nabla _{E_2}e_{3},\frac{\nabla\varphi }{|\nabla \varphi |}\right \rangle 
=\lambda^{-1}.
\end{align*}
Moreover, from the previous paragraph, we have known that 
$Rm^{\lambda}(E_2,E_3,E_2,E_3)=\lambda ^{-2}$. Hence, $K^{\lambda}=0$.

So it follows that each regular level surface $\Sigma$ is isometric to either the flat torus $\mathbb{T}^2$, or 
the cylinder $\mathcal{C}=\mathbb{S}^1\times\mathbb{R}$, or the plane $\mathbb{R}^2$.
Moreover,  from the above computation,  one sees that both $\nabla_{E_3} \mathbf{T}=0$ and 
$\nabla_{E_2} \mathbf{T}=0 \mbox{ mod } E_1$, so $\mathbf{T}$ is parallel on the level 
surface with respect to the metric $g_\Sigma$ induced from $h^\lambda$. Note that the critical set of $\varphi$ cannot contain isolated points because
$\varphi_0=0$ ensures that the trajectory along $\mathbf{T}$ starting 
from a critical point is necessarily contained in the critical set. Hence, each connected component of the critical set is at least $1$-dimensional. On the other hand, 
by Lemma 6.2, the potential function $\varphi$ cannot be locally constant so the critical set of $\varphi $ is at most 2-dimensional. Thus, the critical set of $\varphi$, if non-empty, may consist of only curves diffeomorphic to the real line or the circle, or surfaces diffeomorphic to the torus, the cylinder, or the plane (by Remark \ref{isopara}). 
Furthermore, if the critical set of $\varphi$ is non-empty it could contain a number of surfaces, but at most two curves because regular level surfaces shrinking to a critical curve will force the manifold to ``close up". So we have the following three cases.

\smallskip

\smallskip
\underline{Case 1}. The critical set of $\varphi$ contains no curves. 

\smallskip
In this case, $M$ is diffeomorphic to either $\mathbb{R}^3$, $\mathbb{T}^2\times \mathbb{R}$, or $\mathbb{S}^1\times \mathbb{R}^2$: 

\begin{itemize}
\item When the critical set is empty, in view of Lemma 5.2 and by using the gradient flow of $\varphi$, it follows that $M$ is diffeomorphic to $\Sigma\times \mathbb{R}$, where $\Sigma=\mathbb{T}^2$, $\mathcal{C}$, or $\mathbb{R}^2$.
This family of flat level surfaces is a foliation of $(M,h^\lambda)$.

\item All critical components are $2$-dimensional. 
If one of the critical surfaces is diffeomorphic to the torus $\mathbb{T}^2$, then by Remark \ref{isopara} the regular level surfaces nearby are layers of the tubular neighborhood which move away from the critical $\mathbb{T}^2$ and form a foliation by tori until they encounter other critical surfaces. Note that other critical surfaces must be tori too. When such critical tori occur, there must be regular tori afterwards.  So the regular foliation continues again. No matter how many critical tori are there, each is joined by foliations of regular tori. Thus, $M$ is diffeomorphic to $\mathbb{T}^2\times \mathbb{R}$. Similarly, if one of the ctitical surfaces is the cylinder $\mathcal{C}$, then all critical surfaces are cylinders and $M$ is diffeomorphic to $\mathcal{C}\times\mathbb{R}\approx S^1\times\mathbb{R}^2$; Finally, if one of the critical surfaces is $\mathbb{R}^2$, then all critical surfaces are $\mathbb{R}^2$ and $M$ is diffeomorphic to $\mathbb{R}^3$. 
\end{itemize}

\smallskip
\underline{Case 2}. The critical set of $\varphi$ contains only one curve $\gamma$. 

\smallskip
In this case, $M$ is diffeomorphic to either $\mathbb{R}^3$ or $\mathbb{T}^2\times [0,\infty)$ with $\mathbb{T}^2\times \{0\}$ collapsing to ${\mathbb S}^1$: 

\begin{itemize}
\item
When $\gamma$ is topologically the real line $\mathbb{R}$, nearby regular level surfaces must be the cylinder $\mathcal{C}$ which move away from $\gamma$ and become a foliation until a critical cylinder occurs. 
Similar to Case 1, no matter how many successive critical cylinders are there, $M$ has to be diffeomorphic to $\mathbb{R}^3$ which is $\mathcal{C}\times [0,\infty)$ with $\mathcal{C}\times \{0\}$ collapsing to $\gamma$. The pseudo Gaussian solitons on the Heisenberg group with $\mu\neq 0$ are of this type and one can check that the adapted metric is 
$$h^\lambda = (e^1)^2 + (e^2)^2 + \lambda^{-2}\theta^2
=(E^1)^2 + (E^2)^2 + \lambda^{-2}\theta^2.$$
\item
When $\gamma$ is a closed curve (i.e., topologically a circle), regular level surfaces must be $\mathbb{T}^2$. Therefore, $M$ is diffeomorphic to $\mathbb{T}^2\times [0,\infty)$ with $\mathbb{T}^2\times \{0\}$ collapsing to ${\mathbb S}^1$. 
This can be realized by taking a quotient of the pseudo Gaussian solitons 
on the Heisenberg group such that the critical line becomes the critical circle.
\end{itemize}

\smallskip
\underline{Case 3}. The critical set of $\varphi$ contains two curves $\gamma_1$ and $\gamma_2$. 

\smallskip
In this case, M is diffeomorphic to either ${\mathbb S}^2 \times {\mathbb R}^1$, ${\mathbb S}^3$, or the lens spaces $L(p,q)$ with $1\leq q<p$:

Since the regular level surfaces are tubes over either $\gamma_1$ or $\gamma_2$ and they form tubular neighborhoods of both $\gamma_1$ and $\gamma_2$, the two curves must both be lines or both be closed curves. As before, 
they are joined by either a family of cylinders or a family of tori, no matter whether these surfaces are critical or regular.
\begin{itemize}

\item
When $\gamma_1$ and $\gamma_2$ are lines, regular level surfaces must be $\mathcal{C}$ and 
M is diffeomorphic to ${\mathbb S}^2 \times {\mathbb R}^1$, where the $\gamma_i$ ($i=1,2$) correspond to $\{N\}\times {\mathbb R}^1, \{S\} \times {\mathbb R}^1$ and cylinders are $K\times {\mathbb R}^1$. Here $N, S$ are the north and south poles respectively, and the $K$'s are horizontal circles ${\mathbb S}^2 \cap \{z=\mbox{constant $c$}\}$ in ${\mathbb S}^2$.


\item
When $\gamma_1$ and $\gamma_2$ are closed curves, 
$M$ admits a foliation of flat tori which degenerates at $\gamma_1$ and $\gamma_2$.
In particular, $M$ admits a Heegaard splitting of two solid tori which means that
$M$ is diffeomorphic to ${\mathbb S}^3$, or ${\mathbb S}^2\times {\mathbb S}^1$, or the lens spaces $L(p,q)$ with $1\leq q<p$. 
However, ${\mathbb S}^2\times {\mathbb S}^1$ can be excluded because the first Betti number of a torsion free closed CR manifold must be even (cf. Appendix of \cite{CH} or Theorem 4.3 in \cite{T}).
\end{itemize}

To summarize, $M$ is diffeomorphic to one of the following spacess:  

\medskip
\noindent $\mathbb{R}^3$, $S^3$, $L(p,q)$, $S^2\times \mathbb{R}$, $S^1\times\mathbb{R}^2$, $\mathbb{T}^2\times \mathbb{R}$, $\mathbb{T}^2\times [0,\infty)\mbox{ with }\mathbb{T}^2\times \{0\} \mbox{ collapsing to }S^1$.

\end{proof}
\textbf{The Proof of Corollary \ref{F} :}

\begin{proof}
We have mentioned in Remark \ref{isopara} that the critical set of 
$\varphi$ consists of $\Sigma_+$ and $\Sigma_-$. 
When $\Delta\varphi=\Delta_b\varphi = 2(\mu-W)$ does not change sign, 
$\varphi$ cannot attain both minimum and maximum on $M$ and thus $M$ is non-compact. 

Suppose that $W>\mu$, then $\Delta\varphi<0$ and $\Sigma_-$ must be empty. 
If $\Sigma_+$ is also empty, then by Theorem \ref{E} (i) $M$ must be  
diffeomorphic to $\mathbb{R}^3$ since it is simply connected. 
On ther other hand, if $\Sigma_+$ is non-empty then it has to be connected for otherwise 
$\varphi$ must attain a local minimum somewhere, which is impossible.
Hence the critical set of $\varphi$ has exactly one component.
It follows from this fact, the assmption of $M$ being simply connected, and Theorem \ref{E} (i) (ii) that $M$ is necessarily diffeomorphic to $\mathbb{R}^3$
(note that the critical set of $\varphi$ for $S^2\times\mathbb{R}$ consists of two lines, so it is not allowed here).
Therefore, 
the only non-trivial simply-connected pseudo-gradient CR Yamabe soliton 
with $W>\mu$ must be diffeomorphic to $\mathbb{R}^3$.

The case $W<\mu$ can be proved similarly.
\end{proof}

\section{Appendix}

We include some basic facts about isoparametric functions mentioned in 
Remark \ref{isopara}.
Recall that our isoparametric potential function $\varphi$
satisfies $|\nabla \varphi|= b(\varphi)$ for some function $b: \mathcal{R}\to\mathbb{R}$,
where $\mathcal{R}$ denotes the range of $\varphi$.

\begin{lemma}[cf. Lemma 3 in \cite{wang}]
The only possible singular level sets 
of $\varphi$ are the smooth focal submanifolds  
$$\Sigma_+:= \{ \varphi= \max_{x\in M} \varphi(x)\}\ \mbox{ and }\ \Sigma_-:= \{\varphi = \min_{x\in M} \varphi(x)\}.$$ 
\end{lemma}

\begin{proof} We follow the argument in \cite{wang}. 
Suppose that $b=|\nabla \varphi|=0$ on the level set 
$\Sigma_c:= \{x\in M |\varphi(x) = c\}$ for some critical value $c\in (\inf_{M}\varphi,\sup_{M}\varphi)$. Then, there exists some number $\epsilon>0$ sufficiently small so that $c$ is the only critical value in interval $[c, c+\epsilon]$. 
In particular, we have  
\begin{equation}
\int_c^{c+\epsilon} \frac{1}{|\nabla \varphi|} d\varphi =\lim_{a\to c^{+}} \int_a^{c+\epsilon} \frac{1}{|\nabla \varphi|} d\varphi=  \epsilon.
\end{equation}

On the other hand, since $b\geq 0$ on $\mathcal{R}$ and $c$ is an interior point in $\mathcal{R}$, one has $b'=0$ at $c$ so the function $b$ is locally bounded by its Taylor's reminder around $c$, i.e., $b(x)\leq A(x-c)^2$ for $x\in [c, c+\epsilon]$ and some constant $A>0$. However, this would imply that 
$$\int_c^{c+\epsilon} \frac{1}{|\nabla \varphi|} d\varphi \geq 
\int_c^{c+\epsilon} \frac{1}{A(x-c)^2} d\varphi=\infty,$$
which contradicts (6.1). 
\end{proof}

Suppose that $\varphi$ is locally a constant function on some open set $U$ of $M$, 
then $b$ and $b'$ will vanish at the boundary of $U$, which is impossible as shown
in the proof of Lemma 6.1. Therefore, we have the following lemma (which is used in the proof of Theorem \ref{E}).

\begin{lemma} The non-constant potential function $\varphi$ cannot be locally constant. In particular, its critical set is at most $2$-dimensional. 
\end{lemma}

\end{document}